\documentclass{article}

\usepackage{fullpage}
\usepackage{amsthm}
\usepackage{amsmath}
\usepackage{amssymb}
\usepackage{enumerate}

\usepackage[utf8]{inputenc} 

\usepackage{graphicx}
\usepackage{caption}
\usepackage{subcaption}

\usepackage[justification=centering]{caption}

\newtheorem{theorem}{Theorem}[section]
\newtheorem{proposition}[theorem]{Proposition}

\newtheorem{corollary}[theorem]{Corollary}

\newtheorem{definition}{Definition}[section]
\newtheorem{example}[theorem]{Example}

\begin{document}

	\title{Generalized Normal Ruled Surface of a Curve in the Euclidean 3-space}
	
	\author{Onur Kaya$^{1}$, Mehmet Önder$^{2}$ \\[3mm] {\it $^{1}$ Manisa Celal Bayar University, Department of Mathematics, 45140, Manisa, Turkey} \\ {\it $^{2}$ Delibekirli Village, Tepe Street No.63, 31440, Kırıkhan, Hatay, Turkey} \\{E-mails: $^{1}$ onur.kaya@cbu.edu.tr, $^{2}$ mehmetonder197999@gmail.com} \date{}}
	
	\maketitle
	
		
	\begin{abstract}
		In this study, we define the generalized normal ruled surface of a curve in the Euclidean 3-space $E^3$. We study the geometry of such surfaces by calculating the Gaussian and mean curvatures to determine when the surface is flat or minimal (equivalently, helicoid). We examine the conditions for the curves lying on this surface to be asymptotic curves, geodesics or lines of curvature. Finally, we obtain the Frenet vectors of generalized normal ruled surface and get some relations with helices and slant ruled surfaces and we give some examples for the obtained results.
	\end{abstract}

	\textbf{AMS Classsification:} 53A05, 53A25.
	
	\textbf{Keywords:} Normal ruled surface; minimal surface; helix; slant ruled surface.
	
	
	\section{Introduction}
	Ruled surfaces have an important role in many areas such as architecture, robotics, computer aided geometric design, physics, design problems in spatial mechanism, etc. In 1930, precontraint concrete has been discovered. Then, these surfaces have had an important role in architectural construction and used to construct the spiral stair-cases, roofs, water-towers and chimney-pieces. For instance, Eero Saarinen (1910–1961) used helicoid surface in staircase in General Motor Technical Center in Michigan. He also used ruled surfaces at Yale and M.I.T. buildings. Furthermore, Antonio Gaudí (1852-1926) designed the many pillars of the Sagrada Familia by using hyperbolic hyperboloids. Builder Felix Candela (1910-1997) has made extensive use of cylinders and the most familiar ruled surfaces \cite{Emmer}. Therefore, ruled surfaces have been the focus of study by many mathematicians and different kinds of such surfaces have been defined and studied. One of these kinds is rectifying developable of a curve which defined by Izumiya and Takeuchi as the envelope of the family of rectifying planes of a space curve. They have studied singularities of such surfaces and also given a local classification. They also defined and studied Darboux developable of a space curve whose singularities are given by the locus of the endpoints of modified Darboux vector of the curve \cite{IzumiyaTakeuchislant, IzumiyaKatsumiYamasaki, IzumiyaTakeuchiruled, IzumiyaTakeuchiGeneric, IzumiyaTakeuchigeom}.
	
	Later, Soliman \textit{et al} have made a different definition for the rectifying developable surface \cite{Solimanetal}. They have defined this surface as the surface whose generator line is unit Darboux vector of a space curve. They have obtained that this surface has pointwise 1-type Gauss map of the first kind with a base plane curve if and only if the base curve is a circle or straight line.

    Recently, Önder defined general type of rectifying ruled surfaces as the surface whose rulings always
    lie on the rectifying plane of the base curve \cite{Onderrectifying}. He has obtained many properties of these special ruled surfaces and showed that only the developable rectifying surfaces are the surfaces defined by Izumiya and Takeuchi. 
	
	Furthermore, in \cite{Onder} Önder has defined some new types of ruled surfaces in 3-dimensional Euclidean space which are called  slant ruled surfaces by using the "slant" concept in \cite{IzumiyaTakeuchislant}. Later, Önder and Kaya have given some differential equation characterizations for slant ruled surfaces \cite{KayaOnderchar}.
	
	In this paper, we define a new type of ruled surfaces called generalized normal ruled surfaces in the Euclidean 3-space. We study their Gaussian and mean curvatures, investigate surface curves on generalized normal ruled surfaces and relations with other special ruled surfaces such as slant ruled surfaces.
	
	
	\section{Preliminaries}
	A surface $F$ is called to be a ruled surface if it is drawn by the continuous movement of a straight line along a curve $\alpha$. Such surfaces are parameterized as $F_{(\alpha, q)} (s,u): I \times \mathbb{R} \rightarrow \mathbb{R}^3$, $\vec{F}_{(\alpha, q)} (s,u) = \vec{\alpha} (s) + u \vec{q} (s)$ where $\alpha: I \subset \mathbb{R} \rightarrow \mathbb{R}^3$ is called the base curve and $q: \mathbb{R} \rightarrow \mathbb{R}^3-\{0\}$ is called the ruling. If $q$ is unit, the ruled surface $F_{(\alpha, q)}$ is cylindrical if and only if $\vec{q}\hspace{2pt}' = 0$ and non-cylindrical otherwise, where the prime "\hspace{1pt}$'$\hspace{1pt}" shows derivative with respect to $s$. The curve $c=c(s)$ which satisfies the condition $\langle \vec{c}\hspace{2pt}', \vec{q}\hspace{2pt}' \rangle = 0$ is called the striction curve of ruled surface $F_{(\alpha, q)}$. The striction curve has an important geometric meaning such that if there exists a common perpendicular to two constructive rulings, then the foot of the common perpendicular on the main ruling is called a central point and striction curve is the locus of central points.
	
	The Gauss map or the unit surface normal $U$ of the ruled surface $F_{(\alpha, q)}$ is defined by
	\begin{equation} \label{generalunitsurfacenormal}
		\vec{U} (s,u) = \frac{\frac{\partial \vec{F}_{(\alpha, q)}}{\partial s} \times \frac{\partial \vec{F}_{(\alpha, q)}}{\partial u}}{\lVert \frac{\partial \vec{F}_{(\alpha, q)}}{\partial s} \times \frac{\partial \vec{F}_{(\alpha, q)}}{\partial u} \rVert}.
	\end{equation}
	
	If $\frac{\partial \vec{F}_{(\alpha, q)}}{\partial s} \times \frac{\partial \vec{F}_{(\alpha, q)}}{\partial u} = 0$ for some points $(s_0,u_0)$, then such points are called singular points of the surface. Otherwise, they are called regular points. The ruled surface $F_{(\alpha, q)}$ is called to be developable if the surface normal does not change along a ruling $q=q_0$. Non-developable ruled surfaces are called skew surfaces. A ruled surface $F_{(\alpha, q)}$ is developable if and only if $\det(\vec{\alpha}', \vec{q}, \vec{q}\hspace{2pt}') = 0$.
	
	When $\lVert \vec{q} (s) \rVert = 1$, the vectors $\vec{h} (s) = \frac{\vec{q}\hspace{1pt}'(s)}{\lVert \vec{q}\hspace{1pt}'(s) \rVert}$ and $\vec{a} (s) = \vec{q} (s) \times \vec{h} (s)$ are called central normal and central tangent, respectively. The orthonormal frame $\{ \vec{q}, \vec{h}, \vec{a} \}$ is called the Frenet frame of ruled surface $F_{(\alpha, q)}$ \cite{KargerNovak}.
	
	\begin{definition}
		\textnormal{\cite{Onder}} A ruled surface $F_{(\alpha, q)}$ is called a $q$-slant or $a$-slant (respectively, $h$-slant) ruled surface if its ruling $q$ (respectively, central normal $h$) always makes a constant angle with a fixed direction.
	\end{definition}
	
	The first fundamental form $I$ and second fundamental form $II$ of ruled surface $F_{(\alpha, q)}$ are defined by
	\begin{equation} \label{fundamentalforms}
		\begin{split}
			I &= E ds^2 + 2 F ds du + G du^2\\
			II &= L ds^2 + 2 M ds du + N du^2
		\end{split}
	\end{equation}
	where
	\begin{equation} \label{EFGformulas}
		E = \left\langle \frac{\partial \vec{F}_{(\alpha, q)}}{\partial s}, \frac{\partial \vec{F}_{(\alpha, q)}}{\partial s} \right\rangle, \hspace{6pt} F = \left\langle \frac{\partial \vec{F}_{(\alpha, q)}}{\partial s}, \frac{\partial \vec{F}_{(\alpha, q)}}{\partial u} \right\rangle, \hspace{6pt} G = \left\langle \frac{\partial \vec{F}_{(\alpha, q)}}{\partial u}, \frac{\partial \vec{F}_{(\alpha, q)}}{\partial u} \right\rangle
	\end{equation}
	and
	\begin{equation} \label{LMNformulas}
		L = \left\langle \frac{\partial^2 \vec{F}_{(\alpha, q)}}{\partial s^2}, \vec{U} \right\rangle, \hspace{6pt} M = \left\langle \frac{\partial^2 \vec{F}_{(\alpha, q)}}{\partial s \partial u}, \vec{U} \right\rangle, \hspace{6pt} N = \left\langle \frac{\partial^2 \vec{F}_{(\alpha, q)}}{\partial u^2}, \vec{U} \right\rangle.
	\end{equation}
	
	The Gaussian curvature $K$ and mean curvature $H$ are defined by
	\begin{equation} \label{KandHformulas}
		\begin{split}
			K & = \frac{LN-M^2}{EG-F^2}\\
			H & = \frac{EN-2FM+GL}{2(EG-F^2)}
		\end{split}
	\end{equation}
	respectively \cite{DoCarmo}. An arbitrary surface is called the flat surface if $K=0$ and called minimal if $H=0$ at all points of the surface.

	Helicoid (or right helicoid) is a special kind of ruled surfaces which is generated by a line attached orthogonally to an axis such that the line moves along the axis and also rotates, both at constant speed.
	
	The following theorem is known as Catalan theorem \cite{FomenkoTuzhilin}.
	\begin{theorem} \label{helicoidminimaltheorem}
		Among all ruled surfaces except planes only the helicoid and its fragments are minimal.
	\end{theorem}
		
	\begin{definition} \label{NBsurfacesdefinition}
		Let $\alpha: I \subset \mathbb{R} \rightarrow \mathbb{R}^3$ be a smooth curve and $\{ \vec{T}, \vec{N}, \vec{B} \}$ be its Frenet frame. The ruled surfaces $F_{(\alpha, N)}$ and $F_{(\alpha, B)}$ are called principal normal surface and binormal surface of $\alpha$, respectively, which are given by the parametrizations
		\begin{equation} \label{NBsurfaces}
			\vec{F}_{(\alpha, N)} = \vec{\alpha} \pm u \vec{N}
			\hspace{6pt}
			\textnormal{and}
			\hspace{6pt}
			\vec{F}_{(\alpha, B)} = \vec{\alpha} \pm u \vec{B},			
		\end{equation}
		respectively.
	\end{definition}
	
	
	\section{Generalized Normal Ruled Surfaces}
In this section, we define generalized normal ruled surfaces of a curve in the Euclidean 3-space $E^3$.
	
	\begin{definition} \label{normalruledsurfacedefinition}
		Let $\alpha(s): I \subset \mathbb{R} \rightarrow \mathbb{R}^3$ be a unit speed curve with arclength parameter $s$, Frenet frame $\{ \vec{T}, \vec{N}, \vec{B} \}$, curvature $\kappa$ and torsion $\tau$. The ruled surface $F_{(\alpha, q_n)} (s,u): I \times \mathbb{R} \rightarrow \mathbb{R}^3$ defined by
		\begin{equation} \label{normalruledsurfaceequation}
			\vec{F}_{(\alpha, q_n)} (s,u) = \vec{\alpha} (s) + u \vec{q}_n (s), \hspace{12pt} \vec{q}_n (s) = a_1 (s) \vec{N} (s) + a_2 (s) \vec{B} (s)
		\end{equation}
		is called the generalized normal ruled surface (GNR-surface) of $\alpha$ where $a_1^2 + a_2^2 = 1$ and $a_1, a_2$ are smooth functions of arc-length parameter $s$.
	\end{definition}
	
	From Definition \ref{normalruledsurfacedefinition}, we can easily see that if $a_1=0$ and $a_2= \pm 1$, then the GNR-surface $F_{(\alpha, q_n)}$ becomes binormal surface $F_{(\alpha, B)}$. Similarly, if $a_1=\pm 1$ and $a_2=0$, then the GNR-surface $F_{(\alpha, q_n)}$ becomes principal normal surface $F_{(\alpha, N)}$.
	
	Let $f: I \subset \mathbb{R} \rightarrow \mathbb{R}$ and $g: I \times \mathbb{R} \rightarrow \mathbb{R}$ be smooth functions defined by
	\begin{equation} \label{fandgfunctions}
		\begin{split}
				f(s) & = a_1' (s) a_2 (s) - a_1 (s) a_2' (s) - \tau (s)\\
				g(s,u) & = 1 - u a_1 (s) \kappa (s)
		\end{split}
	\end{equation}
	We will call $f$ and $g$ as the characterization functions of GNR-surface $F_{(\alpha, q_n)}$ and give the many properties of the surface with respect to $f$ and $g$.
	\begin{theorem} \label{regulartheorem}
		The surface $F_{(\alpha, q_n)}$ is not regular if and only if $f=g=0$.
	\end{theorem}
	\begin{proof}
		From the partial derivatives of equation (\ref{normalruledsurfaceequation}) we get
		\begin{equation} \label{nonunitsurfacenormal}
			\begin{split}
				\frac{\partial \vec{F}_{(\alpha, q_n)}}{\partial s} & = g \vec{T} + u (a_1'-a_2 \tau) \vec{N} + u (a_1 \tau + a_2') \vec{B}\\
				\frac{\partial \vec{F}_{(\alpha, q_n)}}{\partial u} & = a_1 \vec{N} + a_2 \vec{B}
			\end{split}
		\end{equation}
		Then, from the cross product of the last equations it follows
		\begin{equation*}
			\frac{\partial \vec{F}_{(\alpha, q_n)}}{\partial s} \times \frac{\partial \vec{F}_{(\alpha, q_n)}}{\partial u} = u f \vec{T} - a_2 g \vec{N} + a_1 g \vec{B}
		\end{equation*}
		and we have that $f=g=0$ if and only if $\frac{\partial \vec{F}_{(\alpha, q_n)}}{\partial s} \times \frac{\partial \vec{F}_{(\alpha, q_n)}}{\partial u} = 0$.
	\end{proof}

This theorem allows to determinate the singular points of the GNR-surfaces as follows.
	
	\begin{proposition} \label{singularpoints}
		If the surface has singular points and $\alpha$ is not a plane curve, for $a_1, a_2 \ne 0$, the locus of the singular points of the GNR-surface $F_{(\alpha, q_n)}$ is given by the curve $\gamma$ defined by
		\begin{equation*}
			\vec{\gamma} (s) = \vec{\alpha} (s) + u (s) \vec{q}_n (s)
		\end{equation*}
		where
		\begin{equation*}
			u (s) = \frac{1}{a_2 \kappa \int \frac{\tau}{a_2^2} ds}.
		\end{equation*}
	\end{proposition}

	\begin{proof}
		From Theorem \ref{regulartheorem}, for the singular points of $F_{(\alpha, q_n)}$ we have
		\begin{equation*}
			\begin{split}
				a_1' (s) a_2 (s) - a_1 (s) a_2' (s) - \tau (s) & = 0,\\
				1 - u a_1 (s) \kappa (s) & = 0.
			\end{split}
		\end{equation*}
		From the first equation of this system, we get $a_1 = a_2 \int \frac{\tau}{a_2^2} ds$ and from the second equation, we get $a_1=\frac{1}{u \kappa}$. By eliminating $a_1$, we obtain the desired result.
	\end{proof}
	
	Since we have $\vec{q}_n (s) = a_1 (s) \vec{N} (s) + a_2 (s) \vec{B} (s)$, by the derivation of the ruling with respect to $s$, it follows
	\begin{equation*}
		\vec{q}_n\hspace{-3.5pt}' = -a_1 \kappa \vec{T} + (a_1' - a_2 \tau) \vec{N} + (a_1 \tau + a_2') \vec{B}.
	\end{equation*}
	Then, the ruling is constant, i.e., $\vec{q}_n\hspace{-3.5pt}' = 0$ if and only if the following system hold
	\begin{equation*}
		\begin{cases}
			a_1 \kappa = 0,\\
			a_1' - a_2 \tau = 0,\\
			a_1 \tau + a_2' = 0.
		\end{cases}
	\end{equation*}
	This system is reduced to
	\begin{equation} \label{cylindrical}
		\begin{cases}
			\kappa = 0,\\
			a_1 = \textnormal{constant},\\
			a_2 = \textnormal{constant}.
		\end{cases}
	\end{equation}
	when $a_1 \ne 0$. In this case, since $\kappa = 0$, we have that the tangent vector $\vec{T}$ is constant. Moreover, the principal normal vector $\vec{N}$ and binormal vector $\vec{B}$ are also constant vectors based on choice which are perpendicular to $\vec{T}$. Since both the tangent of the base curve $\alpha$ and the ruling are constant, the GNR-surface $F_{(\alpha, q_n)}$ becomes a plane. This gives the following corollary:
	
	\begin{corollary} \label{cylindricalcorollary}
		If $a_1 \ne 0$, among all GNR-surfaces $F_{(\alpha, q_n)}$ only the plane is cylindrical.
	\end{corollary}
	
	 In the case that $a_1=0$, the surface becomes binormal surface $F_{(\alpha, B)}$ and we get $\tau = 0$ which means that $\alpha$ is a planar curve. For this case we can give the following corollary
	
	\begin{corollary}
		The binormal surface $F_{(\alpha, B)}$ is cylindrical if and only if $\tau = 0$, i.e., $\alpha$ is a planar curve.
	\end{corollary}
	
	The striction parameter of the GNR-surface $F_{(\alpha, q_n)}$ can be achieved by
	\begin{equation} \label{strictionparameter}
		u(s) = - \frac{\langle \vec{\alpha}', \vec{q}_n\hspace{-3.5pt}' \rangle}{\langle \vec{q}_n\hspace{-3.5pt}', \vec{q}_n\hspace{-3.5pt}' \rangle} = \frac{a_1 \kappa}{a_1^2 \kappa^2 + (a_1' - a_2 \tau)^2 + (a_1 \tau + a_2')^2}.
	\end{equation}
	If $a_1=0$, then we have $\vec{q}_n = \pm \vec{B}$. Thus, $\langle \vec{\alpha}', \vec{q}_n\hspace{-3.5pt}' \rangle = \langle \vec{T}, \mp \tau \vec{N} \rangle = 0$. Therefore, we have the following corollary:
			
	\begin{corollary}
		The base curve $\alpha$ of surface $F_{(\alpha, q_n)}$ is its striction curve if and only if $F_{(\alpha, q_n)}=F_{(\alpha, B)}$ or $\alpha$ is a straight line.
	\end{corollary}

	\begin{proposition} \label{developableproposition}
		The GNR-surface $F_{(\alpha, q_n)}$ is developable if and only if $f=0$.
	\end{proposition}
	\begin{proof}
		A ruled surface $F_{(\alpha, q)}$ is developable if and only if $\det(\vec{\alpha}', \vec{q}, \vec{q}\hspace{2pt}') = 0$. Then, we get
		\begin{equation} \label{developable}
			\begin{split}
				\det(\vec{\alpha}', \vec{q}_n, \vec{q}_n\hspace{-3.5pt}') & = \det(\vec{T}, a_1 \vec{N} + a_2 \vec{B}, -a_1 \kappa \vec{T} + (a_1' - a_2 \tau) \vec{N} + (a_1 \tau + a_2') \vec{B})\\
				& = a_1 a_2'-a_1' a_2 + \tau\\
				& = -f
			\end{split}
		\end{equation}
		From (\ref{developable}), it is clear that $\det(\vec{\alpha}', \vec{q}_n, \vec{q}_n\hspace{-3.5pt}')=0$ if and only if $f=0$.
	\end{proof}
	
	Considering Theorem \ref{regulartheorem} and Proposition \ref{developableproposition}, the following corollary is obtained.
	
	\begin{corollary}
		A developable GNR-surface $F_{(\alpha, q_n)}$ is always regular if and only if $g \ne 0$.
	\end{corollary}

	The Gaussian map or the unit surface normal $U$ of the GNR-surface $F_{(\alpha, q_n)}$ can easily be calculated from (\ref{generalunitsurfacenormal}) as
	\begin{equation} \label{unitsurfacenormal}
		\vec{U} (s,u) = \frac{1}{\sqrt{u^2 f^2 + g^2}} \left( u f \vec{T} - a_2 g \vec{N} + a_1 g \vec{B} \right).
	\end{equation}
	Considering base curve $\alpha$ on $F_{(\alpha, q_n)}$, the unit surface normal $\vec{U}$ can be restricted to $\alpha$ by taking $u=0$, i.e.,
	\begin{equation}
		\vec{U}_\alpha = -a_2 \vec{N} + a_1 \vec{B}.
	\end{equation}
	Then, we can give the followings:
	\begin{proposition}
		(i) The base curve $\alpha$ is a geodesic if and only if $\alpha$ is a straight line or $F_{(\alpha, q_n)}$=$F_{(\alpha, B)}$.
		(ii) The base curve $\alpha$ is an aymptotic curve if and only if $\alpha$ is a straight line or $F_{(\alpha, q_n)}$=$F_{(\alpha, N)}$.
	\end{proposition}
	\begin{proof}
		For the curve $\alpha$ to be a geodesic on $F_{(\alpha, q_n)}$, the directions of principal normal $N$ of $\alpha$ and the surface normal $\vec{U}_\alpha$ along $\alpha$ must be identical. Then, $\alpha$ is a geodesic if and only if $\vec{U}_\alpha \times \vec{\alpha}'' = 0$. Using this equality, we get
		\begin{equation*}
			\vec{U}_\alpha \times \vec{\alpha}'' = \left( -a_2 \vec{N} + a_1 \vec{B} \right) \times \kappa \vec{N} = -a_1 \kappa \vec{T}.
		\end{equation*}
		which gives us (i).

		Similarly, for the base curve $\alpha$ to be an asymptotic curve on $F_{(\alpha, q_n)}$, the principal normal $N$ of $\alpha$ and surface normal $\vec{U}_\alpha$ along $\alpha$ must be perpendicular. Then, $\alpha$ is an asymptotic curve if and only if $\langle \vec{U}_\alpha, \vec{\alpha}'' \rangle = 0$. Using this equality, we get
		\begin{equation*}
			\langle \vec{U}_\alpha, \vec{\alpha}'' \rangle = \langle -a_2 \vec{N} + a_1 \vec{B}, \kappa \vec{N} \rangle = - a_2 \kappa
		\end{equation*}
		which gives (ii).
	\end{proof}
	
	\begin{proposition} \label{lineofcurvature}
		The base curve $\alpha$ is a line of curvature on the GNR-surface $F_{(\alpha, q_n)}$ if and only if the system
		\begin{equation*}
			\begin{cases}
				a_2' + a_1 \tau = 0\\
				a_1' - a_2 \tau = 0
			\end{cases}
		\end{equation*}
		satisfies.
	\end{proposition}
	\begin{proof}
		For the curve $\alpha$ to be a line of curvature on $F_{(\alpha, q_n)}$, the tangent vector $\vec{T}$ of the curve and the derivative of the surface normal along $\alpha$ must be in the same direction, i.e., $\vec{\alpha}' \times \vec{U}_\alpha' = 0$. Then, we get
		\begin{equation*}
			\begin{split}
				\vec{\alpha}' \times \vec{U}_\alpha' & = \vec{T} \times \left( a_2 \kappa \vec{T} - (a_2' + a_1 \tau) \vec{N} + (a_1' - a_2 \tau) \vec{B} \right)\\
				& = - (a_1' - a_2 \tau)\vec{N} - (a_2' + a_1 \tau)\vec{B}
			\end{split}
		\end{equation*}
		Since the vectors $\vec{N}$ and $\vec{B}$ are linearly independent, we obtain that $\vec{\alpha}' \times \vec{U}_\alpha' = 0$ if and only if $a_2' + a_1 \tau = 0$ and $a_1' - a_2 \tau = 0$.
	\end{proof}
	
	Let now consider the Gaussian curvature $K$ and the mean curvature $H$ of GNR-surface $F_{(\alpha, q_n)}$. The fundamental coefficients of GNR-surface $F_{(\alpha, q_n)}$ are calculated from (\ref{EFGformulas}) and (\ref{LMNformulas}) as
	\begin{equation} \label{EFG}
		\begin{split}
			E & = g^2 + u^2 \left[ (a_1' - a_2 \tau)^2 + (a_1 \tau + a_2')^2 \right]\\
			F & = 0\\
			G & = 1
		\end{split}
	\end{equation}
	and
	\begin{equation} \label{LMN}
		\begin{split}
			L & = \frac{1}{\sqrt{u^2 f^2 + g^2}} \left[ -uf\left( g' -u \kappa (a_1'-a_2 \tau) \right) \right.\\
			& \hspace{70pt} \left. + g \left( -a_2 \kappa g + a_2 \tau (a_2'+a_1 \tau) + u \left( -f' + a_1 \tau (a_1'-a_2 \tau) \right)  \right) \right] \\
			M & = \frac{-f}{\sqrt{u^2 f^2 + g^2}}\\
			N & = 0
		\end{split}
	\end{equation}
respectively. Then, from (\ref{KandHformulas}) the Gaussian curvature and mean curvature of GNR-surface $F_{(\alpha, q_n)}$ are
	\begin{equation} \label{KandH}
		\begin{split}
			K & = - \frac{M^2}{E} = - \frac{f^2}{\left( u^2 f^2 + g^2 \right) \left[ g^2 + u^2 \left[ (a_1'-a_2 \tau)^2 + (a_1 \tau + a_2')^2 \right] \right] }\\
			H & = \frac{L}{2E} = \left[ -uf\left( g' -u \kappa (a_1'-a_2 \tau) \right) \right.\\
			& \hspace{45pt} \left. + g \left( -a_2 \kappa g + a_2 \tau (a_1 \tau + a_2') + u \left( -f' + a_1 \tau (a_1'-a_2 \tau) \right)  \right) \right]\\
			& \hspace{45pt} / \left\lbrace 2 \sqrt{u^2 f^2 + g^2} \left[ g^2 + u^2 \left[ (a_1'-a_2 \tau)^2 + (a_1 \tau + a_2')^2 \right] \right] \right\rbrace
		\end{split}
	\end{equation}
	respectively. As we see from (\ref{KandH}) and Proposition \ref{developableproposition}, a classical characterization for developable ruled surfaces stated as "a ruled surface is developable if and only if the Gaussian curvature vanishes" holds for GNR-surfaces.
	
	From (\ref{KandH}), we have the following corollary:

	\begin{corollary}
		Between the Gaussian curvature $K$ and the mean curvatıre $H$ of a GNR-surface $F_{(\alpha, q_n)}$ the following relation exists 
		\begin{equation*}
			KL + 2HM^2 = 0
		\end{equation*}
		or substituting $L$ and $M$ from (\ref{LMN}),
		\begin{equation*}
			\begin{split}
				\frac{K}{\sqrt{u^2 f^2 + g^2}} & \left[ -uf\left( g' -u \kappa (a_1'-a_2 \tau) \right) \right.\\
				& \hspace{10pt} \left. + g \left( -a_2 \kappa g + a_2 \tau (a_1 \tau + a_2') + u \left( -f' + a_1 \tau (a_1'-a_2 \tau) \right)  \right) \right] - \frac{2 H f^2}{u^2 f^2 + g^2} = 0.\\
			\end{split}
		\end{equation*}
	\end{corollary}
	
	\begin{proposition} \label{minimalconditionproposition}
		The GNR-surface $F_{(\alpha, q_n)}$ is minimal if and only if the equality
		\begin{equation} \label{minimalcondition}
			\frac{f}{g} = \frac{-a_2 \kappa g + a_2 \tau (a_2'+a_1 \tau) + u \left( -f' + a_1 \tau (a_1'-a_2 \tau) \right)}{u \left( g' -u \kappa (a_1'-a_2 \tau) \right)}
		\end{equation}
		satisfies.
	\end{proposition}
	\begin{proof}
		The proof is clear from the second equality in (\ref{KandH}). 
	\end{proof}
	
	From Catalan theorem, Theorem \ref{helicoidminimaltheorem},  and Proposition \ref{minimalconditionproposition} we obtain the following corollaries:
	
	\begin{corollary}
		The GNR-surface $F_{(\alpha, q_n)}$ is a plane, a helicoid or its fragment if and only if (\ref{minimalcondition}) holds.
	\end{corollary}
	
	\begin{corollary}
		Let the base curve $\alpha$ of the normal surface $F_{(\alpha, q_n)}$ be a straight line. Then, $F_{(\alpha, q_n)}$ is a minimal surface if and only if $f$ is constant.
	\end{corollary}
		
		
	
	Now, let $F_{(\alpha, q_n)}$ be developable. Then, $f=0$ and we have $a_1'a_2 - a_1 a_2 ' = \tau$. Since $q_n$ is unit, $a_1 a_1' + a_2 a_2' = 0$. Now, it follows,
	\begin{equation*}
		a_1 \tau + a_2' = 0
		\hspace{8pt}
		\textnormal{and}
		\hspace{6pt}
		a_1' - a_2 \tau = 0.
	\end{equation*}
	
	Now, from (\ref{minimalcondition}) we have that $a_2 \kappa g^2 =0$. From the last equations and Proposition \ref{lineofcurvature}, the followings are obtained:
	\begin{corollary}
		The base curve $\alpha$ is always a line of curvature on the developable GNR-surface $F_{(\alpha, q_n)}$.
	\end{corollary}
	
	\begin{corollary}
		A developable regular GNR-surface $F_{(\alpha, q_n)}$ is minimal (or equivalently a helicoid or a plane) if and only if $\alpha$ is a straight line or $F_{(\alpha, q_n)}=F_{(\alpha, N)}$.
	\end{corollary}
	
	Since we assume $F_{(\alpha, q_n)}$ is developable, the partial derivatives given by (\ref{nonunitsurfacenormal}) become
	\begin{equation} \label{developablepartialderivatives}
		\begin{split}
			\frac{\partial \vec{F}_{(\alpha, q_n)}}{\partial s} & = g \vec{T}\\
			\frac{\partial \vec{F}_{(\alpha, q_n)}}{\partial u} & = a_1 \vec{N} + a_2 \vec{B}
		\end{split}
	\end{equation}
	and the unit surface normal $U$ becomes $\vec{U} = - a_2 \vec{N} + a_1 \vec{B}$. By considering the base $\left\lbrace  \frac{\partial \vec{F}_{(\alpha, q_n)}}{\partial s}, \frac{\partial \vec{F}_{(\alpha, q_n)}}{\partial u} \right\rbrace$ of the tangent space $T_p F_{(\alpha, q_n)}$ at a point $p \in F_{(\alpha, q_n)}$, for a vector $v_p \in T_p F_{(\alpha, q_n)}$, the Weingarten map of $F_{(\alpha, q_n)}$ is given by $S_p = - D_p v: T_p F_{(\alpha, q_n)} \rightarrow T_{v_p} S^2$. Then, we have
	\begin{equation*}
		\begin{split}
			S_p \left( \frac{\partial \vec{F}_{(\alpha, q_n)}}{\partial s} \right) = & D_{\frac{\partial F_{(\alpha, q_n)}}{\partial s}} \vec{U}\\
			= & -\frac{\partial \vec{U}}{\partial s}\\
			= & -\frac{\partial }{\partial s} \left( - a_2 \vec{N} + a_1 \vec{B} \right)\\
			= & -a_2 \kappa \vec{T}\\
			= & -\frac{a_2 \kappa}{g} \frac{\partial \vec{F}_{(\alpha, q_n)}}{\partial s}\\
			S_p \left( \frac{\partial \vec{F}_{(\alpha, q_n)}}{\partial s} \right) = & D_{\frac{\partial F_{(\alpha, q_n)}}{\partial u}} \vec{U}\\
			= & -\frac{\partial \vec{U}}{\partial u}\\
			= & 0
		\end{split}
	\end{equation*}
	Thus, the matrix form of the Weingarten map is given as
	\begin{equation} \label{weingartenmapmatrix}
		S_p
		=
		\begin{bmatrix}
			-\frac{a_2 \kappa}{g} & 0 \\
			0 & 0
		\end{bmatrix}
	\end{equation}
	From (\ref{weingartenmapmatrix}), it is easy to see that for a developable GNR-surface $F_{(\alpha, q_n)}$, the Gaussian curvature and mean curvature are
	\begin{equation*}
		K = det(S_p)=0
		\hspace{8pt}
		\textnormal{and}
		\hspace{8pt}
		H = \frac{1}{2}tr(S_p)=-\frac{a_2 \kappa}{2g},
		\hspace{8pt}
	\end{equation*}
	respectively.
	From the characteristic equation $det(S_p-\lambda I)=0$ of the matrix of Weingarten map, the principal curvatures of the developable GNR-surface $F_{(\alpha, q_n)}$ are $\lambda_1 = -\frac{a_2 \kappa}{g}$, $\lambda_2 = 0$ and we have the following corollary:
	
	\begin{corollary}
		\begin{enumerate} [(i)]
			\item If $\kappa \ne 0$ and $a_2 \ne 0$, there are no umbilical points on the surface $F_{(\alpha, q_n)}$.
			\item If $\kappa = 0$ or $a_2 = 0$, then we get $\lambda_1=\lambda_2=0$ and all the points of the surface $F_{(\alpha, q_n)}$ are planar and the quadratic approach of the surface is a plane.
			\item If $\kappa \ne 0$ and $a_2 \ne 0$, then $\lambda_1 \lambda_2 = 0$, $\lambda_1 \ne 0$ and all points of the surface $F_{(\alpha, q_n)}$ are parabolic. Therefore, the quadratic approach of the surface is a parabolic cylinder.
		\end{enumerate}
	\end{corollary}
	
	The principal directions $\vec{e}_1$, $\vec{e}_2$ of developable GNR-surface are defined by $S_p (\vec{e}_1) = \lambda_1 \vec{e}_1$, $S_p (\vec{e}_2) = \lambda_2 \vec{e}_2$ and calculated as
	\begin{equation} \label{principaldirections}
		\begin{split}
			\vec{e}_1 & = g \vec{T}\\
			\vec{e}_2 & =  a_1 \vec{N} + a_2 \vec{B}
		\end{split}
	\end{equation}
	respectively. Since a curve on a surface is a line of curvature if its tangent vector is a principal line, i.e., $S_p (\vec{T}) = m \vec{T}$, $m \in \mathbb{R}$, from the first equality, we have that the base curve $\alpha$ is always a line of curvature on developable GNR-surface $F_{(\alpha, q_n)}$, i.e., we have Corollary 3.13 again. Furthermore, for the parameter curves of the developable normal surface $F_{(\alpha, q_n)}$, we have the following corollary:
	
	\begin{corollary}
		All parameter curves $F_{(\alpha, q_n)} (s, u_0)$ and $F_{(\alpha, q_n)} (s_0, u)$ of a developable GNR-surface $F_{(\alpha, q_n)}$ are lines of curvatures.
	\end{corollary}
	
	Let now $\vec{v}_p$ be a unit tangent vector in the tangent space $T_p F_{(\alpha, q_n)}$ at a point $p$ on developable GNR-surface $F_{(\alpha, q_n)}$. Then, $\vec{v}_p$ can be written as
	\begin{equation} \label{unittangentvector}
		\vec{v}_p = C(s,u) \frac{\partial \vec{F}_{(\alpha, q_n)}}{\partial s} + D(s,u) \frac{\partial \vec{F}_{(\alpha, q_n)}}{\partial u}
	\end{equation}
	where $C,D$ are differentiable functions. Using the linearity of Weingarten map, we have
	\begin{equation*}
		S_p (\vec{v}_p) = C a_2 \kappa \vec{T}
	\end{equation*}
	and substituting (\ref{developablepartialderivatives}) in (\ref{unittangentvector}), we get
	\begin{equation*}
		\vec{v}_p = C g \vec{T} + D \vec{q}_n.
	\end{equation*}
	Since, the normal curvature in the direction of $\vec{v}_p$ is given by $k_n (\vec{v}_p) = \left\langle S_p (\vec{v}_p), \vec{v}_p \right\rangle$, we have
	\begin{equation*}
		k_n (\vec{v}_p) = C^2 a_2 \kappa g
	\end{equation*}
	If $C=0$, $a_2=0$ or $\kappa = 0$, then $k_n (\vec{v}_p) = 0$ and we have the following theorem:
	\begin{theorem}
		\begin{enumerate}[(i)]
			\item If $\kappa \ne 0$, then a unit tangent vector $\vec{v}_p$ on the developable GNR-surface $F_{(\alpha, q_n)}$ is asymptotic if and only if $\vec{v}_p$ is a ruling i.e,  $\vec{v}_p = \vec{q}_n$ or $a_2=0$.
			\item If $\alpha$ is a straight line, then any tangent vector $\vec{v}_p$ is asymptotic.
		\end{enumerate}
	\end{theorem}
	
	
	
	Let $\varphi$ be a regular curve on the developable GNR-surface $F_{(\alpha, q_n)}$ given by the parametrization $\vec{\varphi} (t) = F_{(\alpha, q_n)} \left( s(t), u(t) \right)$. Since the tangent vector of $\varphi$ lies in the tangent space of $F_{(\alpha, q_n)}$, from (\ref{unittangentvector}) we get
	\begin{equation*}
		\dot{\vec{\varphi}} = \frac{d \vec{\varphi}}{dt} = \frac{\partial \vec{F}_{(\alpha, q_n)}}{\partial s} \frac{ds}{dt} + \frac{\partial \vec{F}_{(\alpha, q_n)}}{\partial u} \frac{du}{dt}
	\end{equation*}
	which yields that
	\begin{equation*}
		\frac{ds}{dt} = \dot{s} = C \left( s(t), u(t) \right)
		\hspace{6pt}
		\textnormal{and}
		\hspace{6pt}
		\frac{du}{dt} = \dot{u} = D \left( s(t), u(t) \right).
	\end{equation*}
	From the derivative of $\vec{v}_p$ with respect to $t$, we obtain
	\begin{equation*}
		\dot{\vec{v}}_p = \left( C^2 g_s + 2CD g_u + \dot{C} g \right) \vec{T} + \left( C^2 g \kappa + \dot{D} a_1 \right) \vec{N} + \dot{D} a_2 \vec{B}
	\end{equation*}
	where $g_s = \frac{\partial g}{\partial s}$ and $g_u = \frac{\partial g}{\partial u}$. Since the unit surface normal along $\varphi$ is $\vec{U}_\varphi = -a_2 \vec{N} + a_1 \vec{B}$, then the geodesic curvature of the surface curve $\varphi$ is given by
	\begin{equation} \label{geodesiccurvature}
		\begin{split}
			\kappa_g & = \langle \vec{v}_p, \dot{\vec{v}}_p \times \vec{U}_\varphi \rangle\\
			& = Cg \left( \dot{D} - C^2 g g_u \right) - D \left( C^2 g_s + 2CD g_u + \dot{C} g \right). 
		\end{split} 
	\end{equation}
	From (\ref{geodesiccurvature}), we have the following corollary:
	\begin{corollary}
		A surface curve $\varphi$ on the developable GNR-surface $F_{(\alpha, q_n)}$ is a geodesic if and only if
		\begin{equation*}
			Cg \left( \dot{D} - C^2 g g_u \right) - D \left( C^2 g_s + 2CD g_u + \dot{C} g \right) = 0.
		\end{equation*} 
	\end{corollary}
	The derivative of unit surface normal along $\varphi$ with respect $s$ is given by $\vec{U}'_\varphi = - S_p (\vec{v}_p) = - C a_2 \kappa \vec{T}$ and therefore the geodesic torsion of the surface curve $\varphi$ is given by
	\begin{equation} \label{geodesictorsion}
		\tau_g = \langle \vec{U}'_\varphi, \vec{U}_\varphi \times \vec{v}_p \rangle = CD a_2 \kappa.
	\end{equation}
	From the equation (\ref{geodesictorsion}) we have the following corollary:
	\begin{corollary}
		\begin{enumerate} [(i)]
			\item If $\kappa = 0$, i.e., the base curve $\alpha$ is a straight line, then all surface curves are lines of curvature.
			\item If $F_{(\alpha, q_n)} = F_{(\alpha, N)}$, then all surface curves are lines of curvature.
			\item If $C=0$, $\kappa \ne 0$ and $a_2 \ne 0$, then $\vec{v}_p = \vec{q}_n$ and only rulings are lines of curvature.
			\item  If $D=0$, $\kappa \ne 0$ and $a_2 \ne 0$, then $\vec{v}_p = \vec{T}$ and only the base curve $\alpha$ is a line of curvature.
		\end{enumerate}
	\end{corollary}
	Now, let us consider the Frenet frame of the GNR-surface $F_{(\alpha, q_n)}$. Since the vector $\vec{q}_n (s) = a_1 (s) \vec{N} (s) + a_2 (s) \vec{B} (s)$ is unit, we can take $a_1(s)=\cos (\theta(s))$ and $a_2(s)=\sin (\theta(s))$, i.e., $\vec{q}_n = \cos \theta \vec{N} + \sin \theta \vec{B}$ where $\theta$ is the angle function between $\vec{q}_n$ and $\vec{N}$. Then, the vectors of the Frenet frame $\{\vec{q}_n, \vec{h}, \vec{a}\}$ of the GNR-surface $F_{(\alpha, q_n)}$ are given by
	\begin{equation} \label{qhavectors}
		\begin{split}
			\vec{q}_n & = \cos \theta \vec{N} + \sin \theta \vec{B}\\
			\vec{h} & = \frac{1}{\sqrt{\kappa^2 \cos^2 \theta + z^2}} \left( - \kappa \cos \theta \vec{T} - z \sin \theta \vec{N} + z \cos \theta \vec{B} \right)\\
			\vec{a} & = \frac{1}{\sqrt{\kappa^2 \cos^2 \theta + z^2}} \left( z \vec{T} - \kappa \cos \theta \sin \theta \vec{N} + \kappa \cos^2 \theta \vec{B} \right)
		\end{split}
	\end{equation}
	where $z = \theta' + \tau$. These equalities give the following proposition:
	\begin{proposition}
		For a GNR-surface $F_{(\alpha, q_n)}$, the followings are equivalent.
		\begin{enumerate} [(i)]
			\item The angle function $\theta$ is given with the equality  $\theta = - \int \tau ds$.
			\item The central normal vector $\vec{h}$ of $F_{(\alpha, q_n)}$ coincides with the tangent vector $\vec{T}$ of the curve $\alpha$.
			\item The central tangent vector $\vec{a}$ of $F_{(\alpha, q_n)}$ lies in the normal plane of the curve $\alpha$.
		\end{enumerate}
	\end{proposition}
	\begin{proof}
		If $\theta = - \int \tau ds$, then we have $z=0$. Therefore, the proof is clear from (\ref{qhavectors}).
	\end{proof}
	
	\begin{corollary}
		Let $F_{(\alpha, q_n)}$ be a GNR-surface with frame $\{\vec{q}_n, \vec{h}, \vec{a}\}$ and the angle function $\theta$ between $\vec{q}_n$ and $\vec{N}$ be given by $\theta = - \int \tau ds$. Then, the GNR-surface $F_{(\alpha, q_n)}$ is an $h$-slant ruled surface if and only if $\alpha$ is a general helix.
	\end{corollary}


	\section{Examples}
	
	In this section, we give some examples for the obtain results in the previous section.
	
	\begin{example}
		Let take the $z$-axis as the base curve. Then, $\vec{\alpha}_1 (s) = (0,0,s)$. By choosing $a_1 (s) = \cos s$, $a_2 (s) = \sin s$, we have $\vec{q}_n = (\cos s, \sin s, 0)$ and GNR-surface $F_{1(\alpha_1, q_n)} (s,u) = (u \cos s, u \sin s, s)$ which is a right helicoid given in Figure \ref{fig1}. For this surface, it is obtained $f=-1$ and $g=1$ which gives that $F_{1(\alpha_1, q_n)}$ is a regular and non-developable GNR-surface. The graph of the base curve $\alpha_1$ is also given in Figure \ref{fig1} colored red.
	\end{example}
	\begin{figure}
		\centering
		\includegraphics[width=0.5\linewidth]{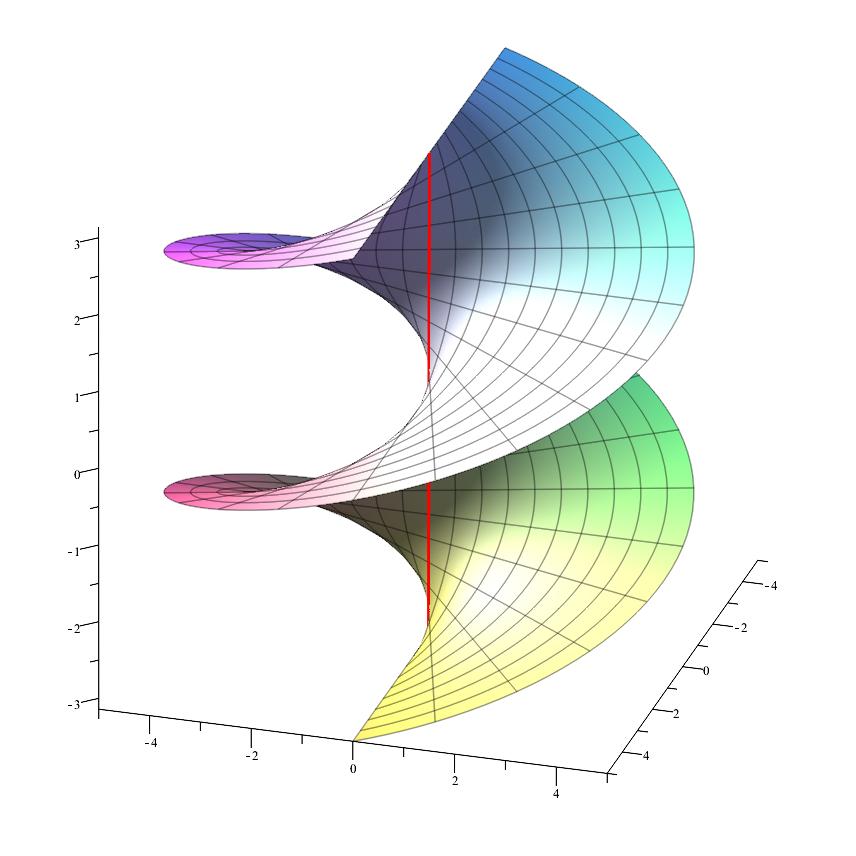}
		\caption{The GNR-surface $F_{1(\alpha_1, q_n)}$}
		\label{fig1}
	\end{figure}
	
	\begin{example}

		Let the GNR-surface $F_{2(\alpha_2, q_n)}$ be given by the parametrization
		\begin{equation*}
			\vec{F}_{2(\alpha_2, q_n)} (s,u)= \vec{\alpha_2} (s) + u \vec{q}_n (s)
		\end{equation*}
		where $\alpha_2(s)$ is a cylidrical helix given by the parametric form
		\begin{equation*}
			\vec{\alpha_2} (s) = \left( \cos \left( \frac{s}{\sqrt{2}} \right), \sin \left( \frac{s}{\sqrt{2}} \right), \frac{s}{\sqrt{2}}  \right)
		\end{equation*}
		The Frenet vectors of $\alpha_2$ are
		\begin{equation*}
			\begin{split}
				\vec{T}(s) & = \left( -\frac{1}{\sqrt{2}} \sin \left( \frac{s}{\sqrt{2}} \right), \frac{1}{\sqrt{2}} \cos \left( \frac{s}{\sqrt{2}} \right), \frac{1}{\sqrt{2}} \right),\\
				\vec{N}(s) & = \left( -\cos \left( \frac{s}{\sqrt{2}} \right), -\sin \left( \frac{s}{\sqrt{2}} \right), 0 \right),\\
				\vec{B}(s) & = \left( \frac{1}{\sqrt{2}} \sin \left( \frac{s}{\sqrt{2}} \right), -\frac{1}{\sqrt{2}} \cos \left( \frac{s}{\sqrt{2}} \right), \frac{1}{\sqrt{2}} \right).
			\end{split}
		\end{equation*}
		By choosing $a_1=\sin \left( \frac{s}{2} \right)$,  $a_2=\cos \left( \frac{s}{2} \right)$, we get
		\begin{equation*}
			\vec{q}_n (s) = \left( -\frac{1}{\sqrt{2}} \cos \left( \frac{s}{\sqrt{2}} \right) + \frac{1}{2} \sin \left( \frac{s}{\sqrt{2}} \right), -\frac{1}{\sqrt{2}} \sin \left( \frac{s}{\sqrt{2}} \right) - \frac{1}{2} \cos \left( \frac{s}{\sqrt{2}} \right), \frac{1}{2} \right)
		\end{equation*}
		and $f=0$, $g=1-\frac{1}{2} u \sin \left( \frac{s}{2} \right)$. For the function $g$ to be zero, we get $u=\frac{2}{\sin \left( \frac{s}{2} \right)}$. Then, the locus of singular points of GNR-surface $F_{2(\alpha_2, q_n)}$ are given by the curve $\gamma (s) = (\gamma_1 (s), \gamma_2 (s) \gamma_3 (s))$ where
		\begin{equation*}
			\begin{split}
				\gamma_1 (s) & = \sqrt{2} \cot \left( \frac{s}{2} \right) \sin \left( \frac{s}{\sqrt{2}} \right) - \cos \left( \frac{s}{\sqrt{2}} \right),\\
				\gamma_2 (s) & = - \sqrt{2} \cot \left( \frac{s}{2} \right) \cos \left( \frac{s}{\sqrt{2}} \right) - \sin \left( \frac{s}{\sqrt{2}} \right),\\
				\gamma_3 (s) & = \frac{s}{\sqrt{2}} + \sqrt{2} \cot \left( \frac{s}{2} \right).
			\end{split}
		\end{equation*}
		The graph of GNR-surface $F_{2(\alpha_2, q_n)}$ is given in Figure \ref{fig2}. In the same figure, the graphs of base curve $\alpha_2$ and locus of singular points of $F_{2(\alpha_2, q_n)}$ are also given by colored red and blue, respectively.
	\end{example}
	
	\begin{figure}
		\centering
		\begin{subfigure}{.5\textwidth}
			\centering
			\includegraphics[width=.9\linewidth]{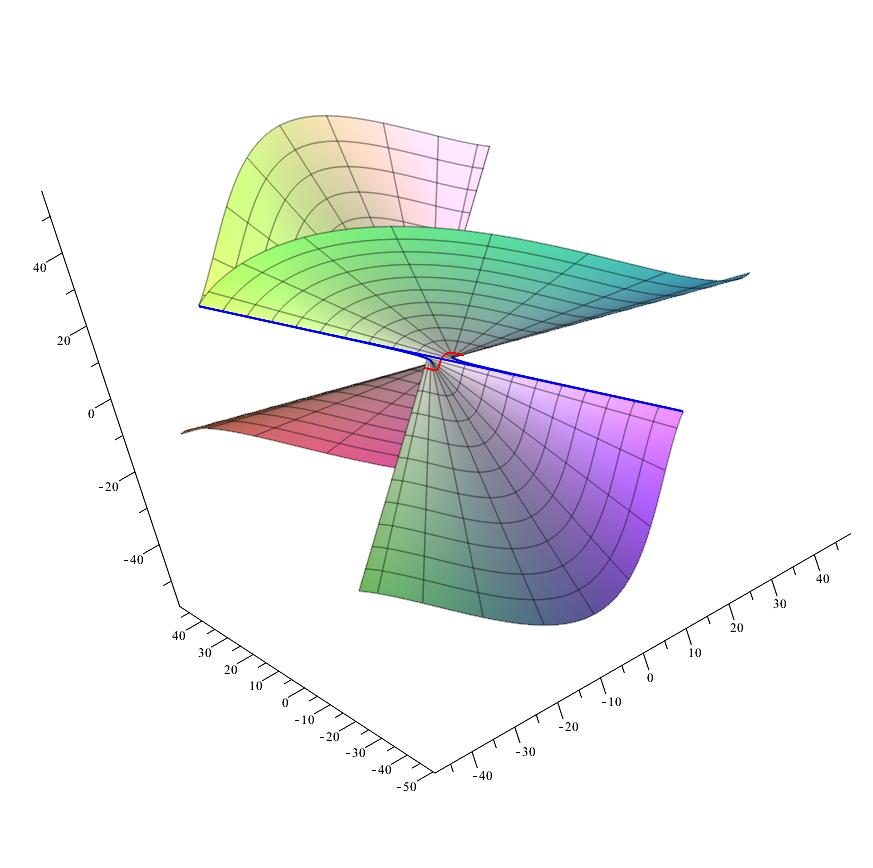}
		\end{subfigure}%
		\begin{subfigure}{.5\textwidth}
			\centering
			\includegraphics[width=.9\linewidth]{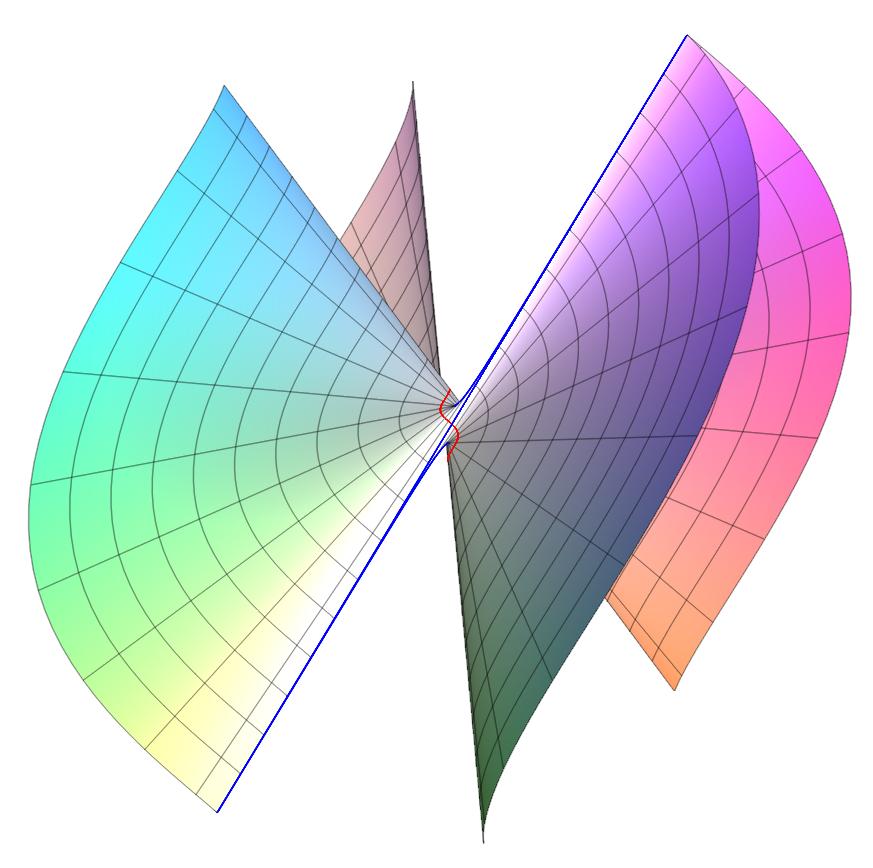}
		\end{subfigure}
		\caption{Two different views of the GNR-surface $F_{2(\alpha_2, q_n)}$}
		\label{fig2}
	\end{figure}

	\begin{example}
		Let the GNR-surface $F_{3(\beta, q_n)}$ be given by the parametrization
		\begin{equation*}
			\vec{F}_{3(\beta, q_n)} (s,u)= \vec{\beta} (s) + u \vec{q}_n (s)
		\end{equation*}
where $\vec{\beta} (s) = (\beta_1 (s), \beta_2 (s), \beta_3 (s))$ is a special curve known as pedal curve with
		\begin{equation*}
			\begin{split}
				\beta_1 (s) & = \frac{3}{2} \cos \left( \frac{s}{2} \right) + \frac{1}{6} \cos \left( \frac{3s}{2} \right),\\
				\beta_2 (s) & = \frac{3}{2} \sin \left( \frac{s}{2} \right) + \frac{1}{6} \sin \left( \frac{3s}{2} \right),\\
				\beta_3 (s) & = \sqrt{3} \cos \left( \frac{s}{2} \right).
			\end{split}
		\end{equation*}
		The Frenet vectors of $\beta$ are given by
		\begin{equation*}
			\begin{split}
				\vec{T} (s) & = \left( -\frac{3}{4} \sin \left( \frac{s}{2} \right) -\frac{1}{4} \sin \left( \frac{3s}{2} \right), \cos^3 \left( \frac{s}{2} \right), -\frac{\sqrt{3}}{2} \sin \left( \frac{s}{2} \right) \right),\\
				\vec{N} (s) & = \left( -\frac{\sqrt{3}}{2} \left( 2 \cos^2 \left( \frac{s}{2} \right) - 1 \right), -\sqrt{3} \cos \left( \frac{s}{2} \right) \sin \left( \frac{s}{2} \right), - \frac{1}{2} \right),\\
				\vec{B} (s) & = \left( \frac{1}{2} \cos \left( \frac{s}{2} \right) \left( 2 \cos^2 \left( \frac{s}{2} \right) - 3 \right), - \sin^3 \left( \frac{s}{2} \right), \frac{\sqrt{3}}{2} \cos \left( \frac{s}{2} \right) \right).
			\end{split}
		\end{equation*}
		If we take $a_1(s) = \cos \left( \frac{s}{2} \right)$ and $a_2 (s) = \sin \left( \frac{s}{2} \right)$, we get $\vec{q}_n (s) = (q_1(s), q_2(s), q_3(s))$ where
		\begin{equation*}
			\begin{split}
				q_1(s) & = \frac{1}{2} \cos \left( \frac{s}{2} \right) \left( 2 \sin \left( \frac{s}{2} \right) \cos^2 \left( \frac{s}{2} \right) -2\sqrt{3} \cos^2 \left( \frac{s}{2} \right) -3 \sin \left( \frac{s}{2} \right) + \sqrt{3} \right),\\
				q_2(s) & = \sin \left( \frac{s}{2} \right) \left( \sin \left( \frac{s}{2} \right) \cos^2 \left( \frac{s}{2} \right) -\sqrt{3} \cos^2 \left( \frac{s}{2} \right) - \sin \left( \frac{s}{2} \right) \right),\\
				q_3(s) & = \frac{1}{2} \cos \left( \frac{s}{2} \right) \left( \sqrt{3} \sin \left( \frac{s}{2} \right) - 1 \right),
			\end{split}
		\end{equation*}
		and $f(s) = \frac{\sqrt{3}}{2} \sin \left( \frac{s}{2} \right) - \frac{1}{2}$, $g(s,u) = 1 - \frac{\sqrt{3}}{2} u \cos^2 \left( \frac{s}{2} \right)$. Assuming $k \in \mathbb{Z}$, for the points
		\begin{equation*}
			(s_0, u_0) = \left( 2 \arcsin \left( \frac{1}{\sqrt{3}} \right) -4k\pi, \frac{2}{\sqrt{3} \cos^2 \left( \arcsin \left( \frac{1}{\sqrt{3}} \right) -2k\pi \right)} \right),
		\end{equation*}
		and
		\begin{equation*}
		(s_1, u_1) = \left( 2 \pi - 2 \arcsin \left( \frac{1}{\sqrt{3}} \right) +4k\pi, \frac{2}{\sqrt{3} \cos^2 \left( \pi - \arcsin \left( \frac{1}{\sqrt{3}} \right) +2k\pi \right)} \right),
		\end{equation*}
		we get $f=g=0$ and therefore these points are singular points on the GNR-surface $F_{3(\beta, q_n)}$. Therefore, except the points $(s_0, u_0)$ and $(s_1, u_1)$, the surface is regular. The graph of the GNR-surface $F_{3(\beta, q_n)}$ is given by Figure \ref{fig3}. The graphs of the base curve $\beta$ and the images of singular points of the surface are also given in Figure \ref{fig3} colored red and blue, respectively.
	\end{example}
	
	\begin{figure}
		\centering
		\begin{subfigure}{.5\textwidth}
			\centering
			\includegraphics[width=.9\linewidth]{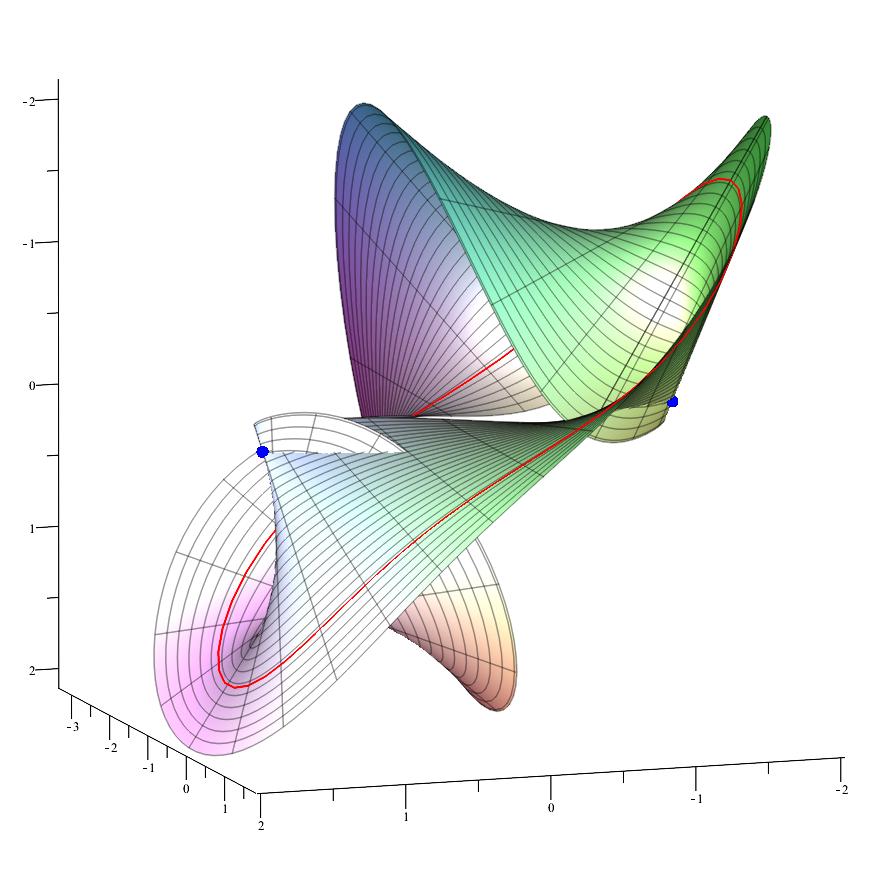}
		\end{subfigure}%
		\begin{subfigure}{.5\textwidth}
			\centering
			\includegraphics[width=.9\linewidth]{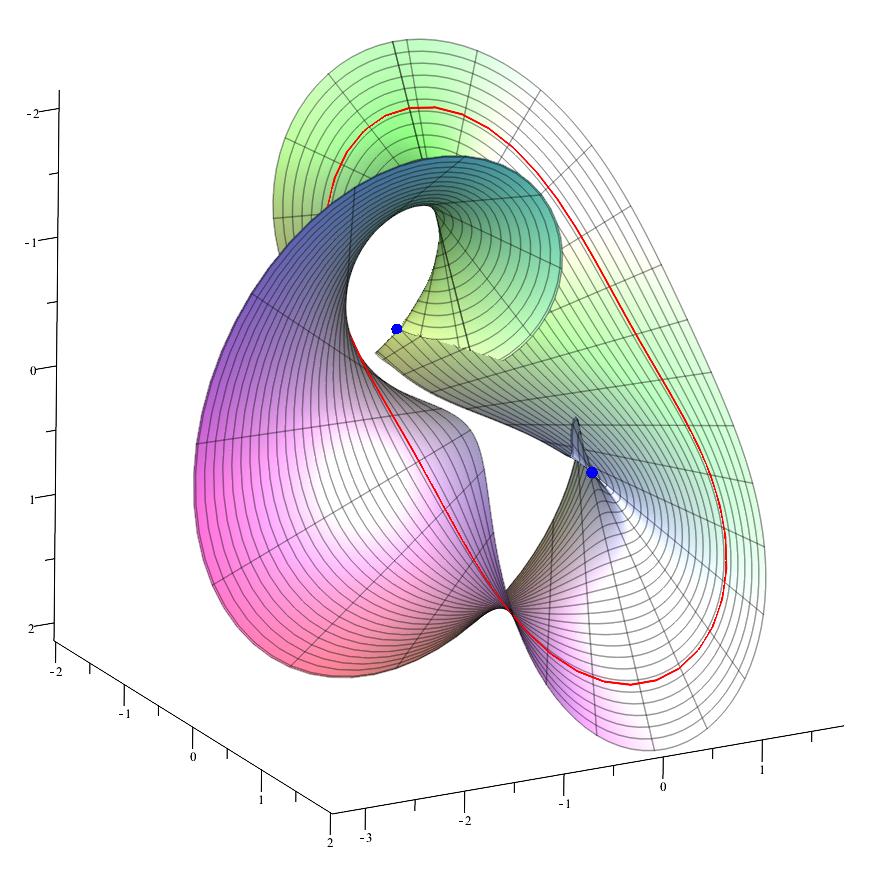}
		\end{subfigure}
		\caption{Two different views of the GNR-surface $F_{3(\beta, q_n)}$}
		\label{fig3}
	\end{figure}

	\begin{figure}
		\centering
		\includegraphics[width=0.5\linewidth]{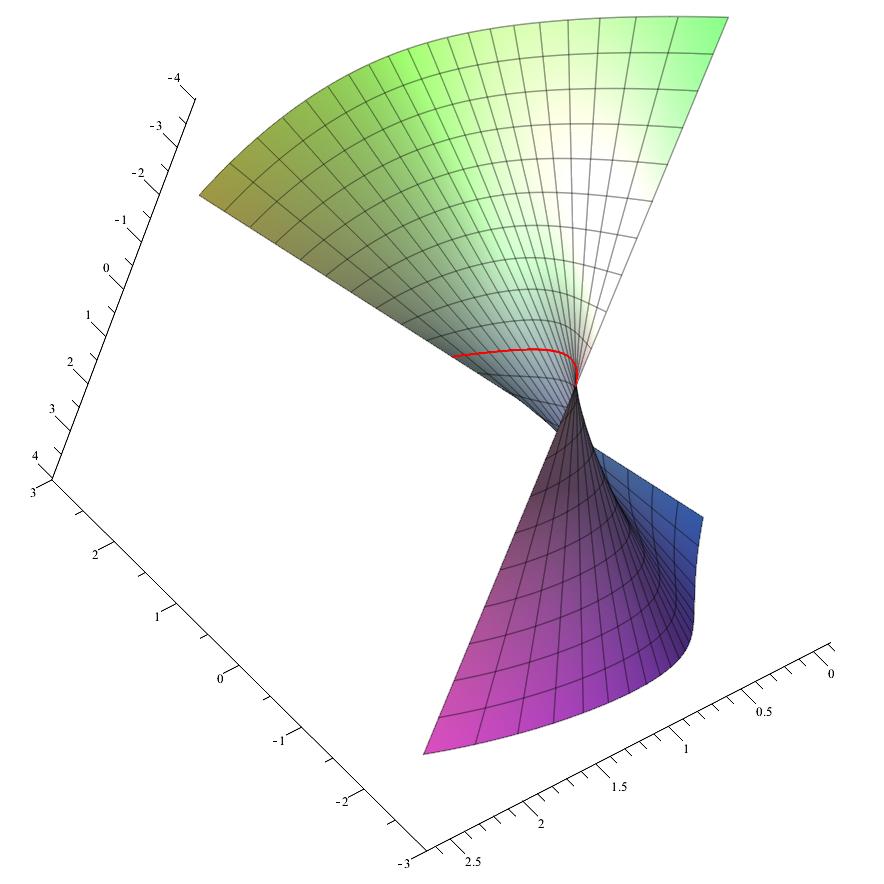}
		\caption{The GNR-surface $F_{4(\alpha_3, q_n)}$}
		\label{fig4}
	\end{figure}
		
	\begin{example}
		Let the GNR-surface $F_{4(\alpha_3, q_n)}$ be given by the parametrization
		\begin{equation*}
			\vec{F}_{4(\alpha_3, q_n)} (s,u)= \vec{\alpha_3} (s) + u \vec{q}_n (s)
		\end{equation*}
		where
		\begin{equation*}
			\vec{\alpha_3} (s) = \left( \sqrt{1+s^2}, s, \ln \left( s + \sqrt{1+s^2} \right) \right).
		\end{equation*}
		The Frenet vectors of $\alpha_3$ are
		\begin{equation*}
			\begin{split}
				\vec{T} & = \left( \frac{\sqrt{2} s}{2 \sqrt{s^2+1}}, \frac{\sqrt{2}}{2}, \frac{\sqrt{2}}{2 \sqrt{s^2+1}} \right),\\
				\vec{N} & = \left( \frac{1}{\sqrt{s^2+1}}, 0, - \frac{s}{\sqrt{s^2+1}} \right),\\
				\vec{B} & = \left( -\frac{\sqrt{2} s}{2 \sqrt{s^2+1}}, \frac{\sqrt{2}}{2}, -\frac{\sqrt{2}}{2 \sqrt{s^2+1}} \right).
			\end{split}
		\end{equation*}
		By assuming $a_1 (s) = s$ and $a_2 (s) = \sqrt{1-s^2}$, we get
		\begin{equation*}
			\vec{q}_n = \left( \frac{s}{\sqrt{s^2+1}} - \frac{\sqrt{2}}{2} \frac{s\sqrt{1-s^2}}{\sqrt{s^2+1}}, \frac{\sqrt{2}}{2} \sqrt{1-s^2}, -\frac{s^2}{\sqrt{s^2+1}} - \frac{\sqrt{2}}{2} \frac{\sqrt{1-s^2}}{s^2+1} \right),
		\end{equation*}
		and
		\begin{equation*}
			f(s) = \frac{2s^2 - \sqrt{1-s^2} + 2}{2 (s^2+1) \sqrt{1-s^2}}, \hspace{6pt} g(s,u)=\frac{2s^2 - u \sqrt{1-s^2} + 2}{2 (s^2+1)}.
		\end{equation*}
		Then, the surface is regular and non-developable. The graph of GNR-surface $F_{4(\alpha_3, q_n)}$ is given by Figure \ref{fig4}.
	\end{example}


\end{document}